\documentclass[a4paper,12pt]{amsart}
\usepackage{amssymb}
\usepackage{ifthen}
\usepackage{graphicx}
\usepackage{float}
\usepackage{caption}
\usepackage{subcaption}
\usepackage{cite}
\usepackage{amsfonts}
\usepackage{amscd}
\usepackage{amsxtra}
\usepackage{color}

\newtheorem{thm}{Theorem}
\newtheorem{cor}{Corollary}
\newtheorem{lem}{Lemma}

\newtheorem{conj}{Conjecture}
\newtheorem{prob}{Problem}

\theoremstyle{definition}
\newtheorem{defn}{Definition}
\newtheorem{example}{Example}

\newtheorem{case}{Case}
\newtheorem{subcase}{Subcase}

\newenvironment{rem}{%
\bigskip
\noindent \textsl{{\sl Remark. }}}{\bigskip}
\newenvironment{rems}{%
\bigskip
\noindent \textsl{{\sl Remarks. }}}{\bigskip}

\newenvironment{pf}[1][]{%
 \vskip 1mm
 \noindent
 \ifthenelse{\equal{#1}{}}%
  {{\slshape Proof. }}%
  {{\slshape #1.} }%
 }%
{\qed\bigskip}

\newcounter{alphabet}

\newenvironment{Thm}[1][]{\refstepcounter{alphabet}%
\bigskip%
\noindent%
{\bf Theorem \Alph{alphabet}}%
\ifthenelse{\equal{#1}{}}{}{ (#1)}%
{\bf .} \itshape}{\vskip 8pt}

\newcommand{\IN}{{\mathbb N}}
\newcommand{\IC}{{\mathbb C}}
\newcommand{\ID}{{\mathbb D}}

\def\be{\begin{equation}}
\def\ee{\end{equation}}

\newcommand{\bee}{\begin{enumerate}}
\newcommand{\eee}{\end{enumerate}}

\newcommand{\blem}{\begin{lem}}
\newcommand{\elem}{\end{lem}}
\newcommand{\bthm}{\begin{thm}}
\newcommand{\ethm}{\end{thm}}
\newcommand{\bcor}{\begin{cor}}
\newcommand{\ecor}{\end{cor}}
\newcommand{\beg}{\begin{example}}
\newcommand{\eeg}{\end{example}}
\newcommand{\begs}{\begin{examples}}
\newcommand{\eegs}{\end{examples}}
\newcommand{\bdefe}{\begin{defn}}
\newcommand{\edefe}{\end{defn}}
\newcommand{\bprob}{\begin{prob}}
\newcommand{\eprob}{\end{prob}}
\newcommand{\bques}{\begin{ques}}
\newcommand{\eques}{\end{ques}}
\newcommand{\bei}{\begin{itemize}}
\newcommand{\eei}{\end{itemize}}
\newcommand{\bcon}{\begin{conj}}
\newcommand{\econ}{\end{conj}}
\newcommand{\bcons}{\begin{conjs}}
\newcommand{\econs}{\end{conjs}}
\newcommand{\bprop}{\begin{propo}}
\newcommand{\eprop}{\end{propo}}
\newcommand{\br}{\begin{rem}}
\newcommand{\er}{\end{rem}}
\newcommand{\brs}{\begin{rems}}
\newcommand{\ers}{\end{rems}}
\newcommand{\bo}{\begin{obser}}
\newcommand{\eo}{\end{obser}}
\newcommand{\bos}{\begin{obsers}}
\newcommand{\eos}{\end{obsers}}

\newcommand{\bca}{\begin{case}}
\newcommand{\eca}{\end{case}}
\newcommand{\bsca}{\begin{subcase}}
\newcommand{\esca}{\end{subcase}}

\newcommand{\bpf}{\begin{pf}}
\newcommand{\epf}{\end{pf}}
\newcommand{\ba}{\begin{array}}
\newcommand{\ea}{\end{array}}
\newcommand{\beq}{\begin{eqnarray}}
\newcommand{\beqq}{\begin{eqnarray*}}
\newcommand{\eeq}{\end{eqnarray}}
\newcommand{\eeqq}{\end{eqnarray*}}

\newcommand{\ra}{\rightarrow}

\newcommand{\ds}{\displaystyle}

\def\cc{\setcounter{equation}{0}   
\setcounter{figure}{0}\setcounter{table}{0}}

\newcounter{minutes}
\divide\time by 60
\newcounter{hours}
\multiply\time by 60 \addtocounter{minutes}{-\time}

\begin{document}

\bibliographystyle{amsplain}
\title[Bohr inequality and  Ces\'aro operators]{The Bohr inequality for the generalized C{\'e}saro averaging operators}

\thanks{
File:~\jobname .tex,
          printed: \number\day-\number\month-\number\year,
          \thehours.\ifnum\theminutes<10{0}\fi\theminutes}


\author{Ilgiz R Kayumov, Diana M. Khammatova and Saminathan Ponnusamy}

\address{I. R Kayumov, Kazan Federal University, 420 008 Kazan, Russia}
\email{ikayumov@gmail.com}

\address{D. M. Khammatova, Kazan Federal University, 420 008 Kazan, Russia}
\email{dianalynx@rambler.ru}

\address{S. Ponnusamy,
Department of Mathematics,
Indian Institute of Technology Madras,
Chennai-600 036, India}
\email{samy@iitm.ac.in}

\subjclass[2000]{Primary: 30A10, 30B10; 30C62, 30H05, 31A05, 41A58; Secondary:  30C75, 40A30}
\keywords{Bohr inequality, Bohr radius, convolution, Gaussian Hypergeometric functions, generalized C{\'e}saro operators.\\
The article is to appear in Mediterranean Journal of Mathematics.
}

\begin{abstract}

The main aim of this paper is to prove a generalization of the classical Bohr theorem and as an application, we obtain a
counterpart of Bohr theorem for the generalized Ces\'aro operator.
\end{abstract}


\maketitle
\pagestyle{myheadings}
\markboth{I. R. Kayumov, D. M. Khammatova and S. Ponnusamy}{Bohr inequality and  Ces\'aro operators}
\cc

\section{Introduction and Preliminaries}

This work is connected with one of classical results known as Bohr's theorem for the class  $\mathcal{B}$ of analytic self mappings of the unit disk
$\mathbb{D}:=\{z\in \mathbb{C}:\,|z|<1 \}$.  Harold Bohr's initial result of 1914
has sharpened by several prominent mathematicians. Since then it has been a source of investigations in numerous other function spaces.
The Bohr theorem in its final form says the following.

\begin{Thm} {\rm \cite[H. Bohr, 1914]{Bohr-14} }\label{KKP16-thA}
If $f\in {\mathcal B}$ and $f(z)=\sum\limits_{n=0}^{\infty} a_n z^n$, then $\sum\limits_{n=0}^{\infty} |a_n|\, r^n\leq 1$
for $r\leq 1/3$ and the constant $1/3$ cannot be improved.	
\end{Thm}

The constant $1/3$ in this theorem is called the Bohr radius. Few other proofs of this result were also given (see \cite{Tomic-62-16}).
It is also true \cite{PaulPopeSingh-02-10} that $\sum_{n=0}^{\infty} |a_n|\, (1/3)^n=1$ if and only if $f$ is a constant function.
However, there are a lot of generalizations and extensions of this theorem (cf. \cite{Bomb-1962,BombBour-2004,DjaRaman-2000,Tomic-62-16}).
The interest on this topic was revived due to the discovery of  extensions to domains in $\IC^n$ and to more general abstract setting
in various contexts, due mainly to works of Aizenberg, Boas, Khavinson, and others (cf. \cite{BDK5,BoasKhavin-97-4,PauSin--09,A2005,A2000,AAD}).
In \cite{A2000,AAD}, multidimensional analogues of Bohr's inequality in which the unit disk $\ID$ is replaced by a domain in $\IC^n$ were considered.
One can also find some information about it, for example, in the survey by Abu-Muhanna et al. \cite{AAPon1}, \cite[Chapter 8]{GarMasRoss-2018} and
the monograph \cite{KM-2007}.

Another widely discussed problem is the investigation of the asymptotical behaviour of the Bohr sum. In this connection, a natural
question is to ask for the best constant $C(r)\geq 1$ such that for $f\in {\mathcal B}$ we have
$$\sum\limits_{n=0}^{\infty} |a_n|\, r^n\leq C(r).
$$
Indeed, Bombieri \cite{Bomb-1962} proved that
$$\sum\limits_{n=0}^{\infty} |a_n|\, r^n \leq \frac{3-\sqrt{8(1-r^2)}}{r} ~\mbox{ for }~ 1/3\leq r\leq 1/\sqrt{2}.
$$
Later in \cite{BombBour-2004}, Bombieri and Bourgain proved that
$$
\sum\limits_{n=0}^{\infty} |a_n|\, r^n< \frac{1}{\sqrt{1-r^{2}}} ~\mbox{ for $r> 1/\sqrt{2}$}
$$
so that $C(r) \asymp (1-r^{2})^{-1/2}$ as $r\to 1$. In the same paper they also obtained a lower bound.
Namely, they proved that for $\varepsilon >0$ there
exists a constant $c=c(\varepsilon)$ such that
$$
\sum\limits_{n=0}^{\infty} |a_n|\, r^n \geq (1-r^2)^{-1/2} - \left( c \log \frac{1}{1-r} \right)^{3/2+\varepsilon} ~\mbox{ as $r\rightarrow 1$}.
$$

Some recent results on the topic can be found in \cite{AliBarSoly, IsmKayPon, KayPon1,KayPon3,LLP2020,LSX2018,PVW2019,PVW201911}.

The article is organized as follows. First, we consider a natural generalization of Theorem~A 
and make it
applicable to many situations (see Theorem \ref{KKP16-th1}). Secondly, in Section \ref{KKP2-sec3} as an application, we investigate a convolution counterpart
of Bohr radius and also the operator counterpart of the so-called one parameter family of averaging Ces\'aro operator $\mathcal{C}_f^{1}$,  discussed for example in \cite{stem94,AHLNP2005}. Finally, in Theorem \ref{KKP2-th5}, we discuss asymptotic Bohr radius for $\mathcal{C}_f^{1}$

\section{Bohr radius in general form}

Let $\{\varphi_k(r)\}_{k=0}^{\infty}$ be a sequence of nonnegative continuous functions in $[0,1)$ such that the series
$\sum_{k=0}^\infty \varphi_k(r)$ converges locally uniformly with respect to $r \in [0,1)$.

\bthm \label{KKP16-th1}
	Let $f\in \mathcal{B}$, $f(z)=\sum_{k=0}^\infty a_k z^k$ and $p\in (0,2]$. If
	\begin{equation}\label{KKP16-eq1}
	\varphi_0(r) > \frac{2}{p} \sum_{k=1}^\infty \varphi_k(r)\qquad \text{for $r\in [0, R),$}
	\end{equation}
where $R$ is the minimal positive root of the equation
$$\varphi_0(x) = \frac{2}{p} \sum_{k=1}^\infty \varphi_k(x),
$$
then  the following sharp inequality holds:
\begin{equation}\label{KKP16-eq2}
	B_{f}(\varphi, p, r):=|a_0|^p \varphi_0(r) + \sum_{k=1}^{\infty} |a_k| \varphi_k(r) \leq \varphi_0(r) ~\mbox{ for all $r \leq R$}.
\end{equation}
In the case when
$  \varphi_0(x) < (2/p)\sum_{k=1}^\infty \varphi_k(x)
$
in some interval $(R,R+\varepsilon)$, the number $R$ cannot be improved. If the functions $\varphi_k(x)$ ($k\geq 0$) are smooth functions then the last condition
is equivalent to the inequality
$$
\varphi_0'(R) <\frac{2}{p} \sum_{k=1}^\infty \varphi_k'(R)
$$
\ethm
\bpf
For $f \in\mathcal{B}$, an application of Schwarz-Pick lemma gives the inequality $|a_k|\leq 1-|a|^2$ for all $k\geq 1$ and thus,
we get that
\beqq
B_{f}(\varphi, p, r) &\leq & |a_0|^p \varphi_0(r) +(1-|a_0|^2)\sum_{k=1}^\infty \varphi_k(r) \\
&= & \varphi_0(r)+(1-|a_0|^2)\left [\sum_{k=1}^\infty \varphi_k(r) - \left (\frac{1-|a_0|^p}{1-|a_0|^2}\right ) \varphi_0(r)\right]\\
&\leq & \varphi_0(r)+(1-|a_0|^2)\left [\sum_{k=1}^\infty \varphi_k(r) -\frac{p}{2}\varphi_0(r)\right]\\
&\leq & \varphi_0(r), ~\mbox{ by Eqn. \eqref{KKP16-eq1},}
\eeqq
for all $r \leq R$,  by the definition of $R$. This proves the desired inequality \eqref{KKP16-eq2}. In the third inequality above,
we have used the following fact:
$$ A(x)=\frac{1-x^p}{1-x^2}\geq \frac{p}{2} ~\mbox{ for all $x\in [0,1)$}
$$
and there is nothing to prove for $p=2$. This inequality is easy to verify. Indeed,
$$ A'(x)=-\frac{xM(x)}{(1-x^2)^2}, \quad M(x)=(2-p)x^p+px^{p-2}-2,
$$
and, since $M'(x)=-p(2-p)x^{p-3}(1-x^{2})< 0$ for $x\in (0,1)$ and for each $p\in (0,2)$, it follows that $M(x)> M(1) =0$
and thus, $A(x)$ is decreasing on $[0,1)$. Hence $A(x)\geq \lim _{x\ra 1^{-}}A(x)=p/2$, as desired.

Now let us prove that $R$ is an optimal number. We consider the function
$$ f(z)=\frac{z-a}{1-az}
$$
with $a \in [0,1)$. For this function we have
\beqq
|a_0|^p \varphi_0(r) + \sum_{k=1}^{\infty} |a_k| \varphi_k(r)
&=&a^p\varphi_0(r)+(1-a^2)\sum_{k=1}^\infty a^{k-1} \varphi_k(r)\\
&=&\varphi_0(r)+(1-a)\left [2 \sum_{k=1}^\infty a^{k-1} \varphi_k(r) - p\varphi_0(r)\right ] \\
&& \quad -(1-a)\left [(1-a) \sum_{k=1}^\infty a^{k-1} \varphi_k(r) + \left ( \frac{1-a^p}{1-a}-p\right )\varphi_0(r)\right ] \\
&=&\varphi_0(r)+(1-a)\left [2 \sum_{k=1}^\infty a^{k-1} \varphi_k(r) - p\varphi_0(r)\right ] +O((1-a)^2)
\eeqq
as $a\ra 1^{-}$. Now it is easy to see that the last number is $>$ $\varphi_0(r)$ when $a$ is close to $1$.
The proof of the theorem is complete.
\epf

\br
Clearly for $p>2$, we see that $1\leq A(x)<p/2$ for $x\in [0,1)$, and thus, in  this case, Theorem \ref{KKP16-th1} holds by replacing the
factor $2/p$ in Eqn. \eqref{KKP16-eq1} by $1$ and also at the other three places in the statement. The most important cases
are at $p=1,2$.
\er
%

\beg\label{KKP16-eg1}
Suppose that $f\in \mathcal{B}$, $f(z)=\sum_{k=0}^\infty a_k z^k$ and $p\in (0,2]$. Then Theorem \ref{KKP16-th1} gives the following:
\bee
\item For $\varphi_k(r)=r^k$ $(k\geq 0)$, we easily have
(see \cite[Proposition 1.4]{Blasco2010} and \cite[Remark 1]{PVW2019})
$$|a_0|^p+\sum_{k=1}^{\infty}|a_k|r^k\leq 1 ~\mbox{for $\ds r\leq R_1(p)=\frac{p}{2+p}$}
$$
and the constant $R_1(p)$ cannot be improved. The case $p=1$ is the classical Bohr inequality. The case $p=2$ is due to \cite{PaulPopeSingh-02-10}
and the inequality in this case does play a special role. We remark that for $p>2$, $R_1(p)$ should be taken as $1/2$.

\item For $\varphi_k(r)=(k+1)r^k$ $(k\geq 0)$, we easily have the sharp inequality
$$|a_0|^p+\sum_{k=1}^{\infty}(k+1)|a_k|r^k\leq 1 ~\mbox{for $\ds r\leq R_2(p)=1-\sqrt{\frac{2}{2+p}}$.}
$$

\item For $\varphi_0(r)=1$ and $\varphi_k(r)=k^{\alpha}r^k$ $(k\geq 1)$, the condition \eqref{KKP16-eq1}  reduces to
$p\geq 2\sum_{k=1}^{\infty}k^{\alpha} r^k$. In particular, as
$$\sum_{k=1}^{\infty}k r^k=\frac{r}{(1-r)^2} ~\mbox{ and }~\sum_{k=1}^{\infty}k^2 r^k=\frac{r(1+r)}{(1-r)^3},
$$
it can be easily seen that the following sharp inequalities (with $\alpha =1, 2$) hold:
$$|a_0|^p+\sum_{k=1}^{\infty}k|a_k|r^k\leq 1 ~\mbox{for $\ds r\leq R_3(p)=\frac{p+1-\sqrt{2p+1}}{p}$}
$$
and
$$|a_0|^p+\sum_{k=1}^{\infty}k^2|a_k|r^k\leq 1 ~\mbox{for $\ds r\leq R_4(p),$}
$$
where $R_4(p)$ is the  minimal positive root of the equation $p(1-r)^3-2r(1+r)=0$.
\eee

\eeg
\section{Convolution counterpart of Bohr radius}\label{KKP2-sec3}

For two analytic functions $f(z)=\sum_{k=0}^\infty a_k z^k$ and $g(z)=\sum_{k=0}^\infty b_k z^k$ in $\ID$, we define the Hadamard product (or convolution)
$f*g$ of $f$ and $g$ by the power series
$$(f*g)(z) =\sum_{k=0}^\infty a_k b_k z^k, \quad z\in \ID.
$$
Clearly, $f*g=g*f$.

As an application of Theorem \ref{KKP16-th1} we consider first the convolution operator of the form
$$(F*f)(z)=\sum_{k=0}^\infty \gamma_k a_k z^k,
$$
where $F(z):={}_2F_1(a,b;c;z)=F(a,b;c;z)$ denotes the Gaussian hypergeometric function defined by the
power series expansion
$$
F(z)= \sum_{k=0}^{\infty} \gamma_k z^k, \quad \gamma_k  = \frac{(a)_k (b)_k}{(c)_k (1)_k}.
$$
Clearly, $F(a,b;c;z)$ is analytic in $\ID$ and in particular, $F(a,1;1;z)=(1-z)^{-a}$.
Here $a,b,c$ are complex numbers such that $c\neq -m$, $m=0,1,2, \ldots$, and
$(a)_k$ is the shifted factorial defined by Appel's symbol
$$(a)_k:=a(a+1)\cdots (a+k-1)=\frac{\Gamma (a+k)}{\Gamma (a)}, \quad k\in\IN  ,
$$
and $(a)_0=1$ for $a\neq 0$. In the exceptional case $c=-m$, $m=0,1,2, \ldots$,
$F(a,b;c;z)$ is defined if $a=-j$ or $b=-j$, where $j=0,1,2, \ldots$ and $j\leq m$. It is clear that if $a=-m$, a negative integer, then
$F(a,b;c;z)$ becomes a polynomial of degree $m$ in $z$.

\bthm\label{KKP16-th2}
Let $f(z)=\sum_{k=0}^\infty a_k z^k$ belong to $\mathcal B$ and $p\in (0,2]$. Assume that  $a,b,c >-1$ such that all $\gamma_k$ have the same sign for $k\geq 0$.
Then
$$
|a_0|^p +\sum_{k=1}^\infty |\gamma_k|\, |a_k| r^k \leq 1 ~\mbox{ for all }~ r \leq R,
$$
where $R$ is the minimal positive root of the equation
$ |F(a,b;c;x)-1|=p/2,
$
and the number $R$ cannot be improved.
\ethm \bpf
We apply Theorem \ref{KKP16-th1}. Set $\varphi_k(r)=\gamma_k r^k$ and remark that $\gamma_0=1$. Let us also note that all $\gamma_k$ have the same sign.
Therefore, we have
$$|F(a,b;c;r)-1|=\sum_{k=1}^\infty \varphi_k(r).
$$
Now the statement of Theorem \ref{KKP16-th1} concludes the proof.
\epf

Sometimes the Bohr radius can be found explicitly. For instance, let us set $b=c=1$. In this case, we have $F(z)=(1-z)^{-a}$ and hence,
$$|F(a,b;c;r)-1|= (1-r)^{-a}-1 =\frac{p}{2}, ~\mbox{ i.e.,}~ R=1-\left(\frac{2}{2+p}\right)^{1/a}
$$
which in the cases $a=1$ and $p=1$ coincide with the classical Bohr radius. Note  that the case $a=2$ is dealt also in
Example \ref{KKP16-eg1}(2).

\section{ $\alpha$-Ces\'aro operators}

For  any $\alpha\in\IC$ with ${\rm Re}\, \alpha>-1$, we consider
$$
\frac{1}{(1-z)^{\alpha+1}}=\sum_{k=0}^{\infty}A_k^{\alpha}z^k,  \quad A_k^{\alpha}=\frac{(\alpha+1)_n}{(1)_n}.
$$
Next, by comparing the coefficient of $z^n$ on both sides of the identity
$$\frac{1}{(1-z)^{\alpha+1}}\cdot\frac{1}{1-z } =\frac{1}{(1-z)^{\alpha+2}},
$$
it follows that
\be\label{KKP16-eq4}
A_n^{\alpha+1}=\sum_{k=0}^{n}A_k^{\alpha}, ~\mbox{ i.e., } \frac{1}{A_n^{\alpha+1}}\sum_{k=0}^{n}A_{n-k}^{\alpha}=1.
\ee
With this principle,
the Ces\'aro operator of order $\alpha$ or $\alpha$-Ces\'aro operator (see Stempak \cite{stem94}) on the space of
analytic functions $f$ in the unit disk $\ID$ is therefore defined by
\be\label{KKP16-eq5}
{\mathcal C}^\alpha f(z)=\sum_{n=0}^{\infty}\left(\frac{1}{A_n^{\alpha+1}}
\sum_{k=0}^{n}A_{n-k}^{\alpha}a_k\right)z^n,
\ee
where $f(z)=\sum_{k=0}^{\infty}a_kz^k$. In terms of convolution, we can write this as
$${\mathcal C}^\alpha f(z)=\frac{f(z)}{(1-z)^{\alpha+1}} *F(1,1;\alpha +2;z)
$$
and thus, we have (cf. \cite{stem94,AHLNP2005})  the following integral form
$${\mathcal C}^\alpha f(z)=(\alpha+1)\int_{0}^{1}f(tz)\frac{(1-t)^\alpha}{(1-tz)^{\alpha+1}}\, dt,
$$
where ${\rm Re}\, \alpha>-1$. This for $\alpha = 0$ gives
$${\mathcal C}^{0}f(z)= \sum_{n=0}^{\infty} \left (\frac{1}{n+1} \sum_{k=0}^{n} a_k \right )z^n,
$$
which is simply the classical Ces\'aro operator  considered by Hardy-Littlewood in 1932 \cite{HaLit32}. Several authors have
studied the boundedness property of these operators on different function spaces (see, for example, \cite{AlBonRic17}). In \cite{KKP2020}, the present authors
established a theorem giving an analog of Bohr theorem for the classical Ces\'aro operator ${\mathcal C}^{0}f$.
It also considered the asymptotical behaviour of the Bohr sum in this case.

For $f\in\mathcal B$, we define a counterpart of Bohr sum for $\alpha >-1$ as
$${\mathcal C}^{\alpha}_f(r) := \sum\limits_{n=0}^{\infty}\left(\frac1{A_n^{\alpha+1}} \sum\limits_{k=0}^{n}A_{n-k}^{\alpha} |a_k| \right) r^n.
$$
Before writing a counterpart of Bohr theorem we need the following estimate.

\bthm
For $f\in\mathcal B$ and $\alpha >-1$, we have
$$\left|{\mathcal C}^{\alpha} (f)(z)\right|\leq (\alpha+1)\Phi(r,1,\alpha+1) = \frac{\alpha+1}{r^{\alpha+1}} \int_0^r\frac{t^{\alpha}}{1-t}\,dt,
$$
where $\Phi(z,s,a) =\sum\limits_{n=0}^{\infty} z^n (n+a)^{-s}$ is the Lerch transcendent function.
\ethm
\bpf
We may represent ${\mathcal C}^{\alpha} (f)$ as
$${\mathcal C}^{\alpha} (f)(z)  =(\alpha+1) \int\limits_0^1\frac1{(t-1)z+1} f\left(\frac{tz}{(t-1)z+1}\right) (1-t)^{\alpha}\, dt
$$
and thus,
\beqq
\left|{\mathcal C}^{\alpha} (f)(z)\right| &\leq& (\alpha+1) \int\limits_0^1\frac{(1-t)^{\alpha}}{|(t-1)z+1|} \left|f\left(\frac{tz}{(t-1)z+1}\right)\right|  dt \\
&\leq& 	(\alpha+1) \int\limits_0^1\frac{(1-t)^{\alpha}}{(t-1)r+1} \,dt = {\mathcal C}_{f_0}^{\alpha}(r), \quad f_0(z)=1.
\eeqq
This integral is not easy to calculate and therefore, we may return to the standard series representation and obtain
$${\mathcal C}_{f_0}^{\alpha}(r)= \sum\limits_{n=0}^{\infty}\frac{A_n^{\alpha}}{A_n^{\alpha+1}}r^n  = (\alpha+1)\sum\limits_{n=0}^{\infty}\frac{r^n}{\alpha+n+1},
$$
and the proof is complete.
\epf

Now we are ready to prove the counterpart of Bohr theorem for the $\alpha$-Ces\'aro operator.

\bthm \label{ThD} Let $f(z)=\sum_{k=0}^\infty a_k z^k$ belong to $\mathcal B$ and $\alpha >-1$. Then
\begin{equation} \label{Eq6}
{\mathcal C}^{\alpha}_f(r) \leq (\alpha+1)\sum_{n=0}^\infty \frac{r^n}{n+\alpha+1} = \frac{\alpha+1}{r^{\alpha+1}} \int_0^r\frac{t^{\alpha}}{1-t}\,dt ~\mbox{ for all }~ r \leq R,
\end{equation}
where $R=R(\alpha )$ is the minimal positive root of the equation
\begin{equation*} 
3 (1+\alpha)\sum_{n=0}^\infty \frac{x^n}{n+\alpha+1}=\frac{2}{1-x}, ~\mbox{ i.e., }~ \sum_{n=0}^\infty \frac{\alpha+1-2n}{n+\alpha+1}x^n =0.
\end{equation*}
The number $R$ cannot be replaced by a larger constant. Note that $R(0) = 0.5335... $.
\ethm \bpf
We apply Theorem \ref{KKP16-th1} with $p=1$.  If we write
\be\label{KKP16-eq6}
{\mathcal C}^{\alpha} (f)(z) =\sum_{n=0}^{\infty}a_n\phi_n(z),
\ee
then collecting the terms involving only $a_n$ in the right hand side of \eqref{KKP16-eq5} we find that
\be\label{KKP16-eq7}
\phi_n(z)= \sum_{k=n}^{\infty}\frac{A_{k-n}^{\alpha}}{A_k^{\alpha+1}}z^k,
\ee
so that for the $\alpha$-Ces\'aro operators ${\mathcal C}^{\alpha} (f)$ we have
$$\varphi_0(x)= \sum_{k=0}^\infty  \frac{A_{k}^{\alpha}}{A_k^{\alpha+1}} x^k
=(\alpha+1)\sum_{k=0}^\infty \frac{x^k}{k+\alpha+1}, \quad x\in [0,1),
$$
by the definition of $A_{k}^{\alpha}$. Moreover, by setting $f(z)=1/(1-z)$ in \eqref{KKP16-eq6}, it is not
difficult to find from \eqref{KKP16-eq4} and \eqref{KKP16-eq5} that
$$\sum_{n=0}^\infty \varphi_n(r) = C^\alpha\left(\frac{1}{1-r}\right)=\frac{1}{1-r}
$$
and thus, Eqn. \eqref{KKP16-eq1} for $p=1$ takes the form
$$3\phi_0(r)> 2 \sum_{n=0}^\infty \varphi_n(r), ~\mbox{ i.e., }~
3 (1+\alpha)\sum_{n=0}^\infty \frac{r^n}{n+\alpha+1}>\frac{2}{1-r}.
$$
The desired inequality (\ref{Eq6}) follows from Theorem \ref{KKP16-th1}
and the sharpness part also follows.
\epf

In what follows ${}_{p}F_q$ represent the generalized hypergeometric function defined by
$$ {}_{p}F_q(a_1, \ldots ,a_p;c_1, \ldots ,c_q;z) =
\sum_{n=0}^{\infty} \frac{(a_1)_n\cdots (a_p)_n}
{(c_1)_n\cdots (c_q)_n}\frac{z^n}{n!}.
$$
We remark that in the interesting case where $p=q+1$, the series converges for
$|z|<1$. If ${\rm Re} \left ( \sum _{j=1}^{q}c_j- \sum _{j=1}^{q+1}a_j \right ) >0$,
then ${}_{q+1}F_q$ converges also at the point $z=1$.

\begin{thm}\label{KKP16-th4}
For $f\in\mathcal B$ and $\alpha >-1$, the inequality
${\mathcal C}^{\alpha}_f(r)\leq S_{\alpha}(r)
$
holds for all $r\in [0,1)$, where $S_{\alpha}(r)$ is equal to
\begin{equation*}
\left\{
\begin{aligned}
\frac{\alpha+1}{1-r^2}\sqrt{\frac{\Gamma''(1+\alpha)}{\Gamma(1+\alpha)} - \left(\frac{\Gamma'(1+\alpha)}{\Gamma(1+\alpha)}\right)^2-r^2\Phi(r^2,2,1+\alpha)} ~\mbox{ for $\alpha\geq0$}\\
\frac{1}{1-r^2} \sqrt{{}_3F_2\left(1,1,1;2+\alpha,2+\alpha;1\right)-r^2{}_3F_2\left(1,1,1;2+\alpha,2+\alpha;r^2\right)}~\mbox{ for $\ds -\frac{1}{2}<\alpha<0$}\\
\frac{1}{1-r}\sqrt{{}_3F_2\left(1,1,1;1.5,1.5;r^2\right)}~\mbox{ for $\ds \alpha=-\frac{1}{2}$}\\
\frac{\Gamma(\alpha+2)}{\Gamma(-\alpha)}	\sqrt{\Gamma(-1-2\alpha)\sum\limits_{n=0}^{\infty} \frac{(n+1)^{1-2\alpha}}{(n+\alpha+1)^2}r^{2n}}~\mbox{ for $\ds -1<\alpha<-\frac{1}{2}$} .
\end{aligned}
\right.
\end{equation*}
Here
$\Phi(z,s,a)$ is the Lerch transcendent function.
\end{thm}
\begin{proof}
First of all, we represent ${\mathcal C}^{\alpha}_f(r)$ as
$$\mathcal{C}_f^{\alpha}(r)=\sum_{n=0}^{\infty}\frac1{A_n^{\alpha+1}}\left(\sum_{k=0}^n A_{n-k}^{\alpha}|a_k|\right)r^n
=\sum_{n=0}^{\infty}|a_n|\phi_n(r),
$$
where $\phi_n(r)$ is defined by \eqref{KKP16-eq7}. 	Using the triangle inequality and the fact that
$\sum_{n=0}^{\infty}|a_n|^2\leq 1$ (for $f\in\mathcal B$), we can estimate
$$\mathcal{C}_f^{\alpha}(r) \leq \sqrt{\sum_{n=0}^{\infty}|a_n|^2}\cdot\sqrt{\sum_{n=0}^{\infty}\phi_n^2(r)} \leq\sqrt{\sum_{n=0}^{\infty}\phi_n^2(r)}.
$$
	
For $\alpha>-1/2$, we use the triangle inequality one more time and obtain
$$\phi_n^2(r)= \left(\sum_{k=n}^{\infty}\frac{A_{k-n}^{\alpha}}{A_k^{\alpha+1}} r^k\right)^2\leq \sum_{k=n}^{\infty}\left(\frac{A_{k-n}^{\alpha}}{A_k^{\alpha+1}}\right)^2\cdot \sum_{k=n}^{\infty} r^{2k}~
=\frac{r^{2n}}{1-r^2}\sum_{k=n}^{\infty}\left(\frac{A_{k-n}^{\alpha}}{A_k^{\alpha+1}}\right)^2.
$$
To estimate the second term on the right, we first observe
that
\be\label{KKP16-eq8}
A_k^{\alpha}=\left (\frac{k+1}{\alpha+k+1}\right )A_{k+1}^{\alpha} ~\mbox{ for all $\alpha >-1$ and $k\geq 0$}.
\ee
First we see that $A_k^{\alpha}\leq A_{k+1}^{\alpha}$ for all $\alpha \geq 0$ and $k\geq 0$. As a consequence,
one can see that for $\alpha>0$ and for $k\geq n$,
$$\frac{A_{k-n}^{\alpha}}{A_k^{\alpha+1}}\leq \frac{A_{k}^{\alpha}}{A_k^{\alpha+1}}=\frac{1+\alpha}{1+k+\alpha}
$$
which gives
$$\phi_n^2(r) \leq  \frac{r^{2n}}{1-r^2} \sum_{k=n}^{\infty}\frac{(\alpha+1)^2}{(1+k+\alpha)^2}
$$
so that
\beq\label{KKP16-eq12}
\sum_{n=0}^{\infty}\phi_n^2(r)
&\leq &\frac{(\alpha+1)^2}{1-r^2}\sum_{n=0}^{\infty} r^{2n}\sum_{k=n}^{\infty} \frac{1}{(1+k+\alpha)^2} \nonumber\\
&=& \frac{(\alpha+1)^2}{1-r^2} \sum_{k=0}^{\infty}\frac{1}{(1+k+\alpha)^2}\sum_{n=0}^kr^{2n}  \nonumber\\
& =&\frac{(\alpha+1)^2}{(1-r^2)^2} \sum_{k=0}^{\infty}\frac{1-r^{2(k+1)}}{(1+k+\alpha)^2}  \nonumber\\
&=&  \frac{(\alpha+1)^2}{(1-r^2)^2}\left(\frac{\Gamma''(1+\alpha)}{\Gamma(1+\alpha)} - \left(\frac{\Gamma'(1+\alpha)}{\Gamma(1+\alpha)}\right)^2-r^2\Phi(r^2,2,1+\alpha)\right).
\eeq

Secondly, by \eqref{KKP16-eq8}, we find that $A_k^{\alpha}\geq A_{k+1}^{\alpha}$ for all $-1<\alpha <0$ and $k\geq 0$, and thus,
\be\label{KKP16-eq13}
\frac{A_{k-n}^{\alpha}}{A_k^{\alpha+1}}\leq \frac{A_{0}^{\alpha}}{A_k^{\alpha+1}}= \frac{1}{A_k^{\alpha+1}}=
\frac{(1)_k}{(\alpha+2)_k}
\ee
which gives
$$\phi_n^2(r) \leq  \frac{r^{2n}}{1-r^2} \sum_{k=n}^{\infty} \left ( \frac{(1)_k}{(\alpha+2)_k}\right )^2.
$$
Hence, as before, we can easily deduce that
\beqq
\sum_{n=0}^{\infty}\phi_n^2(r) &\leq &
\frac{1}{(1-r^2)^2} \sum_{k=0}^{\infty}\left ( \frac{(1)_k}{(\alpha+2)_k}\right )^2 (1-r^{2(k+1)}) \\
&=&  \frac{1}{(1-r^2)^2}\left({}_3F_2\left(1,1,1;2+\alpha,2+\alpha;1\right)-r^2{}_3F_2\left(1,1,1;2+\alpha,2+\alpha;r^2\right)\right)
\eeqq
and the last expression converges for $-0.5<\alpha<0$.

Thirdly, for $\alpha=-0.5$ we use the inequality \eqref{KKP16-eq13} and obtain
$$\frac{A_{k-n}^{-1/2}}{A_k^{1/2}}\leq \frac{A_{0}^{-1/2}}{A_n^{1/2}}= \frac{1}{A_n^{1/2}}=  \frac{(1)_n}{(3/2)_n} =\frac{\sqrt{\pi}\Gamma(n+1)}{2\Gamma(n+3/2)}.
$$

In this case
\beqq
\sum\limits_{n=0}^{\infty}\phi_n^2(r) &\leq & \sum\limits_{n=0}^{\infty}\frac{\pi\Gamma^2(n+1)}{4\Gamma^2(n+1.5)}\left(\sum\limits_{k=n}^{\infty}r^k\right)^2
= \frac{\pi}{4(1-r)^2}\sum\limits_{n=0}^{\infty}\frac{\Gamma^2(n+1)}{\Gamma^2(n+1.5)}r^{2n} \\
&= & \frac{1}{(1-r)^2}{}_3F_2\left(1,1,1;1.5,1.5;r^2\right).
\eeqq
	
Finally, for $\alpha<-0.5$, we write the quotient as
$$\frac{A_{k-n}^{\alpha}}{A_k^{\alpha+1}}=  \frac{(\alpha+1)_{k-n}}{(1)_{k-n}} \cdot \frac{(1)_k}{(\alpha+2)_k}
$$
and estimate $\phi_n^2(r)$ using the triangle inequality in the following way:
\beqq
\phi_n^2(r) & =&  \left ( \sum_{k=n}^{\infty}\frac{(\alpha+1)_{k-n}}{(1)_{k-n}} \cdot \frac{(1)_k}{(\alpha+2)_k} r^k\right )^2\\
& \leq &   \sum_{k=n}^{\infty}  \left ( \frac{(\alpha+1)_{k-n}}{(1)_{k-n}}\right )^2  \sum_{k=n}^{\infty}  \left ( \frac{(1)_k}{(\alpha+2)_k} \right )^2r^{2k}\\
& = &    {}_2F_1\left(\alpha +1,\alpha +1;1;1\right) \sum_{k=n}^{\infty}  \left ( \frac{(1)_k}{(\alpha+2)_k} \right )^2r^{2k},
\eeqq
where the first sum in the second step converges to
$${}_2F_1\left(\alpha +1,\alpha +1;1;1\right)= \frac{\Gamma(-1-2\alpha)}{\Gamma^2(-\alpha)}
~\mbox{ for $-1<\alpha <-0.5$.}
$$
Here we have used the well-known formula
$$ {}_2F_1(a,b;c;1) = \frac{\Gamma (c)\Gamma (c-a-b)}{\Gamma (c-a)\Gamma (c-b)}< \infty ~\mbox{ for $c>a+b$}.
$$
To estimate the other series, we use the well-known inequality
$$\frac{(1)_n}{(s)_n}\leq \frac{\Gamma (s)}{(n+1)^{s-1}},
$$
which holds for any natural $n$ and $0\leq s\leq 1$. Therefore,
$$\frac{(1)_k}{(\alpha+2)_k}= \left ( \frac{\alpha +1}{\alpha+k+1)}\right ) \frac{(1)_k}{(\alpha+1)_k} \leq  \frac{\Gamma (\alpha +2)}{\alpha+k+1)}  \frac{1}{(k+1)^{\alpha}}.
$$
It follows that
\beqq
\sum_{n=0}^{\infty}\phi_n^2(r)&\leq& {}_2F_1\left(\alpha +1,\alpha +1;1;1\right) \Gamma^2(\alpha+2)\sum\limits_{n=0}^{\infty}\sum_{k=n}^{\infty} \frac{(k+1)^{-2\alpha}}{(k+\alpha+1)^2}r^{2k} \\
& \leq &\frac{\Gamma(-1-2\alpha)\Gamma^2(\alpha+2)}{\Gamma^2(-\alpha)}\sum\limits_{n=0}^{\infty} \frac{(n+1)^{1-2\alpha}}{(n+\alpha+1)^2}r^{2n}.
\eeqq

This finishes the proof.
\end{proof}

\section{Asymptotic Bohr radius for $\mathcal{C}_f^{1}$}

Let us study the order of the estimate for $\alpha=1$. We recall the following equality:
$$
{\rm Li\,}_2(x)+{\rm Li\,}_2(1-x) = \frac{\pi^2}6-\log x\log(1-x).
$$
Then the estimate looks like
$$\frac{1}{1-r^2} \cdot 2\sqrt{\frac{\pi^2}6-\frac{{\rm Li\,}_2(r^2)}{r^2}} = \frac{\ds 2\sqrt{\frac{\pi^2}6-\frac{1}{r^2}\left(\frac{\pi^2}6-2\log r\log(1-r^2)-{\rm Li\,}_2(1-r^2)\right)}}{1-r^2},
$$
where ${\rm Li\,}_2(z) = \sum_{k=1}^{\infty}(1/k^2)z^k$ is a polylogarithm function.
Moreover,
$${\rm Li\,}_2(1-r^2)\to0 ~\mbox{ and }~\frac{\log (1/r)}{1-r}\to1 ~\mbox{ as $r\to1$,}
$$
and so we obtain
$$
\mathcal{C}_f^{1}(r)\leq \frac{2\sqrt{2\log r\log(1-r^2)}+o(1)}{1-r^2} \sim \sqrt{2}\cdot\sqrt{\frac{\log\frac1{1-r^2}}{1-r}} .
$$

We shall prove

\bthm \label{KKP2-th5}
There exists an $f\in {\mathcal B}$ such that
$$
\mathcal{C}_f^{1}(r)\sim \frac{4\sqrt{2q}}{(3+q)\sqrt{1-r}}\approx \frac{1.47217\dots}{\sqrt{1-r}},
$$
where $q\approx 7.57736\dots$ is the root of  the equation $3q = (3+q)\log(1+q).$
\ethm
\bpf
Consider the function $\phi_n(r)$ defined by \eqref{KKP16-eq7} with $\alpha =1$ and calculate it using definition of
$A_k^{\alpha}$. This gives that
$$\phi_n(r) = \sum_{k=n}^{\infty}\frac{A_{k-n}^{1}}{A_k^{2}}r^k = 2 \sum_{k=n}^{\infty} \frac{k-n+1}{(k+2)(k+1)}r^k.
$$
As
$$\frac{r^{k+2}}{(k+2)(k+1)} =\int_0^r\int_0^{\rho}s^k\,ds\,d\rho,
$$
it follows that
\be\label{KKP16-eq10}
\phi_n(r) =  \frac2{r^2} \int\limits_0^r \int\limits_0^{\rho} s^n \sum_{k=n}^{\infty} (k-n+1) s^{k-n}\,ds\,d\rho  =
\frac2{r^2}\int\limits_0^r\int\limits_0^{\rho}\frac{s^n}{(1-s)^2}\,ds\,d\rho .
\ee

In \cite{BombBour-2004}, Bombieri and Bourgain showed how one can build functions $h(z) = \sum_{k=0}^{\infty}h_kz^k$ and
$f(z)= \sum_{k=0}^{\infty}a_kz^k$ in the family $\mathcal{B}$ such that
\bee
\item[{\rm (i)}] $|h_k|=t^k\sqrt{1-t^2}$ \quad {\rm (ii)}
$\ds \|f-h\|_2 :=  \sqrt{\sum_{k=0}^{\infty}|h_k-a_k|^2} < \sqrt{1-t^2}\sqrt{\log\frac 1{1-t}}$,
\eee
where $0\leq t\leq1$ is some number.

Accordingly, using their idea, we obtain the following estimate
\beq\label{KKP16-eq11}
\mathcal{C}_f^{1}(r) & =&  \sum_{k=0}^{\infty}|h_k+a_k-h_k|\phi_k(r)
\geq
\sum_{k=0}^{\infty}|h_k|\phi_k(r) -\sum_{k=0}^{\infty}|h_k-a_k|\phi_k(r)\nonumber\\\
&\geq & \sqrt{1-t^2}\sum_{k=0}^{\infty}t^k\phi_k(r) - \sqrt{1-t^2}\sqrt{\log\frac 1{1-t}}\sqrt{\sum_{k=0}^{\infty}\phi_k^2(r)}.
\eeq
Using \eqref{KKP16-eq10} and the above consideration, the first term in \eqref{KKP16-eq11} can be calculated directly. Note that
$$\sum_{k=0}^{\infty}t^k\phi_k(r) = \frac{2}{r^2}\sum_{k=0}^{\infty} \int\limits_0^r\int\limits_0^{\rho}\frac{t^k s^k}{(1-s)^2}\,ds\,d\rho
 =\frac{2}{r^2}\int\limits_0^r\int\limits_0^{\rho}\frac{ds\,d\rho}{(1-s)^2(1-ts)}
$$
and so to computet the integral on the right, we write
$$\frac{1}{(1-s)^2(1-ts)} =\frac{1}{1-t}\cdot \frac{1}{(1-s)^2}   -\frac{t}{(1-t)^2}\cdot \frac{1}{1-s} +\frac{t^2}{(1-t)^2}\cdot\frac{1}{1-ts},
$$
which gives by integration
$$ \int\limits_0^{\rho}\frac{ds}{(1-s)^2(1-ts)} = \frac{1}{1-t}\cdot \frac{\rho }{1-\rho}   +\frac{t}{(1-t)^2}\left [\log (1-\rho) -\log (1-t\rho)\right ]
$$
and hence we can easily obtain by integrating it again
$$\int\limits_0^r\int\limits_0^{\rho}\frac{ds\,d\rho}{(1-s)^2(1-ts)} = \frac{2[-r(1-t)+(1-rt)(\log(1-rt)-\log(1-r))]}{r^2(1-t)^2}.
$$

The second term in \eqref{KKP16-eq11} can be estimated using the result from the first part of Theorem \ref{KKP16-th4}, i.e., Eqn. \eqref{KKP16-eq12}.
For $\alpha=1$, we have
$$\sqrt{\sum_{k=0}^{\infty}\phi_k^2(r)} \leq \frac{2}{1-r^2}\sqrt{\frac{\Gamma''(2)}{\Gamma(2)} - \left(\frac{\Gamma'(2)}{\Gamma(2)}\right)^2-r^2\Phi(r^2,2,2)}
= \frac{2}{1-r^2} \sqrt{\frac{\pi^2}6-\frac{{\rm Li\,}_2(r^2)}{r^2}}.
$$

Therefore, the above discussion shows that  $\mathcal{C}_f^{1}(r)$ is not less than
\begin{multline*}
\frac1{\sqrt{1-r}}\left(\frac{2\sqrt{1-t^2}\sqrt{1-r}\,[-r(1-t)+(1-rt)(\log(1-rt)-\log(1-r))]}{r^2(1-t)^2} - \right.\\- \left.
2\sqrt{1-t^2}\sqrt{\log\frac 1{1-t}}\cdot \frac{1 }{\sqrt{1-r}(1+r)} \sqrt{\frac{\pi^2}6-\frac{{\rm Li\,}_2(r^2)}{r^2}}\right).
\end{multline*}

Let $t=r^q$ for some $q$. Since
$$
\lim\limits_{r\to 1}\sqrt{1-r^{2q}}\sqrt{\log\frac 1{1-r^q}}\cdot \frac{1}{\sqrt{1-r}} \sqrt{\frac{\pi^2}6-\frac{{\rm Li\,}_2(r^2)}{r^2}} = 0,
$$
we have to consider only the first term
$$
\frac{1}{\sqrt{1-r}}\cdot \frac{2\sqrt{1+r}\,[-r+(1+r)\log(1+r)]}{r^2} \sim \frac{2\sqrt{2}(\log 4-1)}{\sqrt{1-r}} \approx \frac{1.09261\dots}{\sqrt{1-r}}.
$$
Next, we calculate the limit
\begin{multline*}
\lim\limits_{r\to1}\frac{2\sqrt{1-r^{2q}}\sqrt{1-r}\left [-r(1-r^q)+(1-r^{q+1})(\log(1-r^{q+1})-\log(1-r))\right ]}{r^2(1-r^q)^2}   \\
= \frac{-q+(q+1)\log(q+1)}{\left(q/2\right)^{\frac32}}.
\end{multline*}
The point of maximum is the root of the equation
$3q = (3+q)\log(1+q).
$
Calculation shows that $q\approx 7.57736\dots$. Therefore,
$$
\mathcal{C}_f^{1}(r)\sim \frac{4\sqrt{2q}}{(3+q)\sqrt{1-r}}\approx \frac{1.47217\dots}{\sqrt{1-r}}.
$$
\epf

\subsection*{Acknowledgments}
The work of    I. Kayumov and D. Khammatova is supported by the Russian
Science Foundation under grant 18-11-00115. 

%

\end{document}